\documentclass[12pt,reqno]{amsart}

\usepackage{amssymb}
\usepackage{amsthm}
\usepackage{amsmath}
\usepackage{a4wide}
\usepackage{latexsym}
\usepackage{mathrsfs}
\usepackage{euscript}
\usepackage[v2,tips]{xy}
\usepackage{float}
\usepackage[T1]{fontenc}
\usepackage[latin1]{inputenc}

\newtheorem{theorem}{Theorem}[section]
\newtheorem{lemma}[theorem]{Lemma}
\newtheorem{proposition}[theorem]{Proposition}
\newtheorem{corollary}[theorem]{Corollary}

\theoremstyle{definition}
\newtheorem{definition}[theorem]{Definition}

\newtheorem{remark}[theorem]{Remark}

\numberwithin{equation}{section}

\newcounter{smallromans}

\newenvironment{romanenumerate}
{\begin{list}{{\normalfont\textrm{(\roman{smallromans})}}}%
    {\usecounter{smallromans}\setlength{\itemindent}{0cm}%
      \setlength{\leftmargin}{5.5ex}\setlength{\labelwidth}{5.5ex}%
      \setlength{\topsep}{0.2ex}\setlength{\partopsep}{0ex}%
      \setlength{\itemsep}{0.2ex}}}%
  {\end{list}}

\newcommand{\romanref}[1]{{\normalfont\textrm{(\ref{#1})}}}

\newcounter{smallalphs}

  {\end{list}}

\renewcommand{\le}{\ensuremath{\leqslant}}

\newcommand{\N}{\mathbb{N}}

\newcommand{\R}{\mathbb{R}}
\newcommand{\C}{\mathbb{C}}
\newcommand{\K}{\mathbb{K}}

\newcommand{\ann}{\operatorname{ann}} 
 
\newcommand{\spa}{\operatorname{span}}
\newcommand{\clspa}{\overline{\operatorname{span}}\,}

\renewcommand{\phi}{\ensuremath{\varphi}}
\renewcommand{\epsilon}{\ensuremath{\varepsilon}}

\newcommand{\smashw}[2][l]{{\text{\makebox[0pt][#1]{$#2$}}}}

\begin{document}
\title[A singular, admissible extension which splits algebraically,
 not strongly]{A singular, admissible extension which splits
  algebraically, but not strongly, of the algebra\\ of bounded
  operators on a Banach space}
\dedicatory{In memoriam: Charles J.~Read (1958--2015)}
\subjclass[2010]%
{Primary 46H40, 
46M18, 
47L10; 
Secondary 16S70, 
46B45, 
46H10} 
\author[T.~Kania]{Tomasz Kania}
\thanks{Kania acknowledges with thanks funding from the European
  Research Council/ERC Grant Agreement no.~291497.}
\address{Mathematics Institute, University of Warwick, Gibbet Hill
  Rd., Coventry CV4 7AL, United Kingdom}
\email{tomasz.marcin.kania@gmail.com}
\author[N.~J.~Laustsen]{Niels Jakob Laustsen} 
\address{Department of Mathematics and Statistics, Fylde College,
  Lancaster University, Lancaster LA1 4YF, United Kingdom.}
\email{n.laustsen@lancaster.ac.uk}
\author[R.~Skillicorn]{Richard Skillicorn} 
\address{Department of Mathematics and Statistics, Fylde College,
  Lancaster University, Lancaster LA1 4YF, United Kingdom.}
\email{r.skillicorn@lancaster.ac.uk}
\begin{abstract}
Let $E$ be the Ba\-nach space constructed by Read (\emph{J.~London
  Math.\ Soc.}~1989) such that the Banach algebra~$\mathscr{B}(E)$ of
bounded operators on~$E$ admits a discontinuous derivation.  We show
that~$\mathscr{B}(E)$ has a singular, admissible extension which
splits algebraically, but does not split strongly.  This answers a
natural question going back to the work of Bade, Dales, and Lykova
(\emph{Mem.\ Amer.\ Math.\ Soc.}~1999), and complements recent results
of Laustsen and Skillicorn (\emph{C.~R.~Math.\ Acad.\ Sci.\ Paris,} to
appear).
\end{abstract}
\maketitle
\section{Introduction and the main result}%
\label{section1}
\noindent 
An \emph{extension} of a Banach algebra~$\mathscr{B}$ is a short-exact
sequence of the form
\begin{equation}\label{shortexactseq} \spreaddiagramcolumns{2ex}%
    \xymatrix{\{0\}\ar[r] & \ker\phi\ar[r] &
      \mathscr{A}\ar^-{\displaystyle{\phi}}[r] & \mathscr{B}\ar[r] &
      \{0\},}
\end{equation} 
where~$\mathscr{A}$ is a Banach algebra
and~$\phi\colon\mathscr{A}\to\mathscr{B}$ is a continuous, surjective
algebra homomorphism. The ex\-ten\-sion is \emph{singular}
if~$\ker\phi$ has the trivial product, that is, $ab=0$ for each
$a,b\in\ker\phi$.  We say that the extension \emph{splits
  algebraically} (respectively, \emph{splits strongly, is admissible})
if~$\phi$ has a right in\-verse which is an algebra homo\-mor\-phism
(respectively, is a continuous algebra homomorphism, is bounded and
linear).  Every extension which splits strongly is evidently
admissible and splits algebraically, and the admissibility
of~\eqref{shortexactseq} is equivalent to $\ker\phi$ being
complemented in~$\mathscr{A}$ as a Banach space.
\smallskip

Motivated by Bade, Dales, and Lykova's comprehensive study~\cite{bdl}
of extensions of Banach algebras, the second- and third-named
authors~\cite{LS1} investigated the interrelationship among the above
properties for extensions of Banach algebras of the
form~$\mathscr{B}(E)$, that is, all bounded operators acting on a
Banach space~$E$, showing that there exist Banach spaces~$E_1$ and
$E_2$ such that:
\begin{romanenumerate}
\item\label{LS1result1} $\mathscr{B}(E_1)$ has a singular extension
  which splits algebraically, but the extension is not admissible, and
  so it does not split strongly;
\item\label{LS1result2} $\mathscr{B}(E_2)$ has a singular extension
  which is admissible, but it does not split algebraically.
\end{romanenumerate}\smallskip

\noindent
These results naturally raise  the following question: 
\begin{center}
\emph{Does there exist a Banach space~$E$ such that~$\mathscr{B}(E)$
  has an admissible\\ extension which splits algebraically, but 
  not strongly?}
\end{center}
The purpose of this note is to answer this question positively. 

\begin{theorem}\label{mainthm}
There exist a Banach space~$E$ and a continuous, surjective algebra
homomorphism~$\phi$ from a unital Banach algebra~$\mathscr{A}$
onto~$\mathscr{B}(E)$ such that the extension
\begin{equation}\label{mainthmEq1} 
  \spreaddiagramcolumns{2ex}\xymatrix{\{0\}\ar[r] & \ker\phi\ar[r] &
    \mathscr{A}\ar^-{\displaystyle{\phi}}[r] & \mathscr{B}(E)\ar[r] &
    \{0\}} \end{equation} is singular and admissible and splits
algebraically, but it does not split strongly.
\end{theorem}

In fact, the Banach spaces $E_1$ and $E_2$ used to establish
statements~\romanref{LS1result1} and~\romanref{LS1result2} above are
equal (but the extensions are obviously distinct), both being the
Banach space~$E_{\text{R}}$ con\-structed by Read~\cite{read} with the
property that there is a discontinuous derivation
from~$\mathscr{B}(E_{\text{R}})$ into a one-dimensional
$\mathscr{B}(E_{\text{R}})$-bimodule. This Banach space will also play
the role of~$E$ in Theorem~\ref{mainthm}.

It is no coincidence that the Banach algebra~$\mathscr{B}(E)$
associated with the Banach space~$E$ in Theorem~\ref{mainthm} admits a
discontinuous derivation.  Indeed, a Banach algebra~$\mathscr{B}$
which has an extension that splits algebraically, but not strongly,
obviously admits a discontinuous algebra homomorphism. This property
is generally weaker than admitting a discontinuous derivation.
However, if in addition we require that the extension be singular and
admissible, then~$\mathscr{B}$ must admit a discontinuous derivation.

To see this, suppose contrapositively that every derivation
from~$\mathscr{B}$ into a Banach $\mathscr{B}$\nobreakdash-bi\-module
is continuous. Then, by a theorem of Dales and Villena (see
\cite[Corollary~2.2]{DV}, or \cite[Corollary 2.7.7]{dales} for an
exposition), every intertwining map (as defined on p.~\pageref{p4}
below) from~$\mathscr{B}$ into a Banach $\mathscr{B}$-bi\-module is
continuous, and therefore, using \cite[Theorem~2.13]{bdl} or
\cite[Theorem~2.8.16]{dales}, we conclude that every singular,
admissible extension of~$\mathscr{B}$ which splits algebraically also
splits strongly.  This result implies in particular that every singular,
admissible extension of the compact
operators~$\mathscr{K}(E_{\text{R}})$ which splits algebraically also
splits strongly, in contrast to Theorem~\ref{mainthm}, because Read
showed that~$E_{\text{R}}$ has a Schauder
basis~\cite[Lemma~5.3]{read}, and hence every derivation
from~$\mathscr{K}(E_{\text{R}})$ into a
$\mathscr{K}(E_{\text{R}})$-bimodule is continuous (see
\cite[Theorem~5.2]{read} or \cite[Corollary~5.3.9]{dales}).\smallskip

The usefulness of Read's Banach space~$E_{\text{R}}$ in the context of
extensions originates from the following theorem, which strengthens
and clarifies the main technical results underlying Read's
construction of a discontinuous derivation
from~$\mathscr{B}(E_{\text{R}})$ (see \cite[p.~306 and
  Sec\-tion~4]{read}). In it, we denote by $\ell_2^{\sim}$ the
unitization of the Hilbert space~$\ell_2$ endowed with the trivial
product, so that $\ell_2^{\sim} = \ell_2\oplus\K 1$ as a
vector space (where~$\K$ denotes the scalar field, either~$\R$
or~$\C$, and~$1$ is the formal identity that we adjoin), and the
product and norm on~$\ell_2^{\sim}$ are given by
\[ (\xi + s1)(\eta + t1) =
s\eta + t\xi+ st1\quad \text{and}\quad \|\xi + s1\| = \|\xi\| +
|s|\qquad (\xi,\eta\in\ell_2,\, s,t\in\mathbb{K}). \]

\begin{theorem}[{\cite[Theorem~1.2]{LS2}}]\label{WEBEsplitexact}
There exists a continuous, surjective algebra
homomorphism~$\psi\colon\mathscr{B}(E_{\normalfont{\text{R}}})\to\ell_2^{\sim}$
with $\ker\psi = \mathscr{W}(E_{\normalfont{\text{R}}})$ (the ideal of
weakly compact operators on~$E_{\normalfont{\text{R}}}$) such that the
extension
\begin{equation}\label{WEBEsplitexactEq1} 
 \spreaddiagramcolumns{2ex}\xymatrix{\{0\}\ar[r] &
   \mathscr{W}(E_{\normalfont{\text{R}}})\ar[r] &
   \mathscr{B}(E_{\normalfont{\text{R}}})\ar^-{\displaystyle{\psi}}[r]
   & \ell_2^{\sim}\ar[r] & \{0\}}
\end{equation} 
splits strongly.
\end{theorem}

\section{Proof of Theorem~\ref{mainthm}} 
\noindent Throughout, all Banach spaces and algebras are considered
over the same scalar field~$\K$, where either $\K = \R$ or
$\K=\C$.

\begin{definition} Let $X$ and $Y$ be Banach spaces,
and let $S\colon X\to Y$ be a linear map. The \emph{separating space}
of~$S$ is given by
\[ \mathfrak{S}(S) = \{ y\in Y: \text{there is a null
  sequence}\ (x_n)_{n\in\N}\ \text{in}\ X\ \text{such that}\ Sx_n\to
y\ \text{as}\ n\to\infty\}. \]
\end{definition}

We shall require the following fundamental facts about the separating
space; see, \emph{e.g.,} \cite[Propositions~5.1.2 and
  5.2.2(ii)]{dales} or \cite[Lemmas 1.2(i)--(ii)
  and~1.3(ii)]{sinclair}. The second clause is of course just a
restatement of the Closed Graph Theorem. It explains why the
separating space plays an important role in automatic continuity
theory.

\begin{proposition}\label{sepspacebasics}  Let $X$ and $Y$ be Banach spaces,
and let $S\colon X\to Y$ be a linear map. Then:
\begin{romanenumerate}
\item\label{sepspacebasics1} $\mathfrak{S}(S)$ is a closed subspace of~$Y;$
\item\label{sepspacebasics2} $S$ is bounded if and only if
  $\mathfrak{S}(S)=\{0\}$.
\item\label{sepspacebasics3} Let $Z$ be a Banach space, and suppose
  that $T\colon Y\to Z$ is bounded and linear. Then \[
  \mathfrak{S}(TS) = \overline{T[\mathfrak{S}(S)]}. \]
\end{romanenumerate}
\end{proposition}

\subsection*{The pullback of Banach algebras} As in \cite{LS1}, 
the connection between the extension~\eqref{WEBEsplitexactEq1}
of~$\ell_2^{\sim}$ and the extension~\eqref{mainthmEq1}
of~$\mathscr{B}(E_{\normalfont{\text{R}}})$ that we aim to establish
goes via the pullback construction, which we shall now review. Let
$\alpha\colon\mathscr{A}\to\mathscr{C}$ and
$\beta\colon\mathscr{B}\to\mathscr{C}$ be continuous algebra
homomorphisms between Banach algebras~$\mathscr{A}$, $\mathscr{B}$,
and~$\mathscr{C}$, and set
\begin{equation}\label{expldefnpullback} 
\mathscr{D} = \{(a,b)\in\mathscr{A}\oplus\mathscr{B} : \alpha(a) =
\beta(b) \}, \end{equation} where~$\mathscr{A}\oplus\mathscr{B}$
denotes the Banach-algebra direct sum of~$\mathscr{A}$
and~$\mathscr{B}$.  For later reference, we note that~$\mathscr{D}$
is unital in the case where~$\mathscr{A}$, $\mathscr{B}$,
$\mathscr{C}$, $\alpha$, and~$\beta$ are unital. Let~$\gamma$
and~$\delta$ be the restrictions to~$\mathscr{D}$ of the coordinate
projections:
\begin{equation}\label{defngammadelta}
 \gamma\colon\ (a,b)\mapsto a,\quad
 \mathscr{D}\to\mathscr{A},\qquad\text{and}\qquad
 \delta\colon\ (a,b)\mapsto b,\quad
 \mathscr{D}\to\mathscr{B}. \end{equation} They are continuous
algebra homomorphisms, and the triple $(\mathscr{D},\gamma,\delta)$
has the uni\-ver\-sal property of pullbacks, as observed in
\cite[Section~2]{LS1}, where the following connection with extensions
is also established.

\begin{proposition}[{\cite[Proposition~2.2]{LS1}}]%
\label{pullbacksandsplittingsofextensions}
Let~$\mathscr{A}$, $\mathscr{B}$, and~$\mathscr{C}$ be Banach algebras
such that there are extensions
\begin{equation}\label{ext1} \spreaddiagramcolumns{2ex}%
    \xymatrix{\{0\}\ar[r] & \ker\alpha\ar[r] &
      \mathscr{A}\ar^-{\displaystyle{\alpha}}[r] & \mathscr{C}\ar[r] &
      \{0\}}
\end{equation} 
and 
\begin{equation}\label{ext2} \spreaddiagramcolumns{2ex}%
    \xymatrix{\{0\}\ar[r] & \ker\beta\ar[r] &
      \mathscr{B}\ar^-{\displaystyle{\beta}}[r] & \mathscr{C}\ar[r] &
      \{0\}\smashw{,}}
\end{equation}
and define $\mathscr{D}$, $\gamma$, and $\delta$
by~\eqref{expldefnpullback} and~\eqref{defngammadelta}
above. Then~$\delta$ is surjective, and the following statements
concerning the extension
\begin{equation}\label{ext3} \spreaddiagramcolumns{2ex}%
    \xymatrix{\{0\}\ar[r] & \ker\delta\ar[r] &
      \mathscr{D}\ar^-{\displaystyle{\delta}}[r] & \mathscr{B}\ar[r] &
      \{0\}}
\end{equation}
 hold true:
\begin{romanenumerate}
\item\label{pullbacksandsplittingsofextensions1} \eqref{ext3} is
  singular if and only if~\eqref{ext1} is singular.
\item\label{pullbacksandsplittingsofextensions2} Suppose
  that~\eqref{ext2} splits strongly (respectively, splits
  algebraically, is admissible).  Then \eqref{ext3} splits strongly
  (respectively, splits algebraically, is admissible) if and only
  if~\eqref{ext1} splits strongly (respectively, splits algebraically,
  is admissible).
\end{romanenumerate}
\end{proposition}

\subsection*{Derivations and intertwining maps}  
Throughout this paragraph, $\mathscr{A}$ denotes a Banach algebra, and
$X$ is a Banach $\mathscr{A}$-bimodule.  The \emph{annihilator}
of~$\mathscr{A}$ in~$X$ is given by
\[ \ann_{\mathscr{A}}X = \{ x\in X: a\cdot x = 0 = x\cdot
a\ \text{for each}\ a\in\mathscr{A}\}. \] Let $S\colon\mathscr{A}\to
X$ be a linear map, and consider the bilinear map
\begin{equation}\label{coboundDefn}
\delta^1 S\colon\ (a,b)\mapsto a\cdot (Sb) - S(ab) + (Sa)\cdot b,\quad
\mathscr{A}\times\mathscr{A}\to X. \end{equation} We say that~$S$ is a
\emph{derivation} if $\delta^1S = 0$, while~$S$ is an
\emph{intertwining map}\label{p4} if $\delta^1S$ is continuous.

Using~$\delta^1S$, we can define an algebra product on the
vector space $\mathscr{A}\oplus X$ as follows:
\begin{equation}\label{AoplusXmult} (a,x)(b,y) = (ab, a\cdot y +
  x\cdot b + (\delta^1S)(a,b))\qquad (a,b\in\mathscr{A},\, x,y\in
  X). \end{equation} Suppose that~$S$ is an intertwining map.  Then
the above product is continuous with respect to the norm
$\|(a,x)\| = \|a\|+\|x\|$ for $a\in\mathscr{A}$ and $x\in X$, and so
$\mathscr{A}\oplus X$ is a Banach algebra with respect to an
equivalent norm. The map
\[ \pi_{\mathscr{A}}\colon (a,x)\mapsto a,\quad\mathscr{A}\oplus
X\to\mathscr{A}, \] 
is clearly a continuous, surjective algebra homomorphism with kernel
$\{0\}\oplus X$, which has the trivial product, so that we
obtain a singular extension of~$\mathscr{A}$: 
\begin{equation}\label{extAoplusX} 
\spreaddiagramcolumns{2ex}\xymatrix{\{0\}\ar[r] & \{0\}\oplus X\ar[r]
  & \mathscr{A}\oplus X\ar^-{\displaystyle{\pi_{\mathscr{A}}}}[r] &
  \mathscr{A}\ar[r] & \{0\}.} \end{equation} Moreover, this extension
is admissible because $a\mapsto
(a,0),\ \mathscr{A}\to\mathscr{A}\oplus X,$ is a bounded, linear right
inverse of~$\pi_{\mathscr{A}}$, and it splits algebraically because
$a\mapsto (a,-Sa),\ \mathscr{A}\to\mathscr{A}\oplus X,$ is an algebra
homomorphism which is also a right inverse of~$\pi_{\mathscr{A}}$. A
direct calculation shows that~\eqref{extAoplusX} splits strongly if
and only if there is a derivation $D\colon \mathscr{A}\to X$ such that
$S-D$ is continuous. This is a well-known generalization of results in
pure algebra, going back to Johnson \cite[Theorem~2.1 and
  Corollary~2.2]{johnsonWed}; see also \cite[Theorem~2.8.12]{dales}
and the text preceding it.\smallskip

As we have already indicated, our strategy is to construct an
extension with suitably chosen properties of the unitization
of~$\ell_2$ endowed with the trivial product, and then apply
Proposition~\ref{pullbacksandsplittingsofextensions} together with
Theorem~\ref{WEBEsplitexact} to obtain the desired extension
of~$\mathscr{B}(E_{\text{R}})$. Adjoining an identity to an extension
is straightforward, in the sense that it does not alter any of the
properties that we are interested in. Consequently, we shall
henceforth focus our attention on the case where the Banach
algebra~$\mathscr{A}$ has the trivial product. It turns out that we
may also simplify matters by supposing that the bimodule~$X$ has the
trivial right action; that is, \mbox{$ab = 0$} and \mbox{$x\cdot a
  =0$} for each $a,b\in\mathscr{A}$ and $x\in X$. In this case,
derivations and intertwining maps have nice characterizations in terms
of the annihilator and the separating space.

\begin{lemma}\label{lemmaIntDeriv}
Let $\mathscr{A}$ be a Banach algebra with the trivial product, let
$X$ be a Banach $\mathscr{A}$-bimodule with the trivial right action,
and let $S\colon\mathscr{A}\to X$ be a linear map. Then:
\begin{romanenumerate}
\item\label{lemmaIntDeriv1} $S$ is a derivation if and only if
  $S[\mathscr{A}]\subseteq\ann_{\mathscr{A}}X;$
\item\label{lemmaIntDeriv2} $S$ is an intertwining map if and only if
  $\mathfrak{S}(S)\subseteq \ann_{\mathscr{A}}X.$
\end{romanenumerate}
\end{lemma}

\begin{proof} We begin by noting that, by the hypotheses,
  \eqref{coboundDefn} reduces to
\begin{equation}\label{lemmaIntDerivEq1} (\delta^1S)(a,b) = a\cdot(Sb)\qquad
  (a,b\in\mathscr{A}),
\end{equation}
from which~\romanref{lemmaIntDeriv1} follows immediately.

\romanref{lemmaIntDeriv2}. Suppose that~$S$ is an intertwining map,
and let $y\in\mathfrak{S}(S)$ and $a\in\mathscr{A}$ be given. Since
$y\cdot a = 0$ by hypothesis, it remains to show that $a\cdot y=0$.
Take a null sequence $(x_n)_{n\in\N}$ in~$\mathscr{A}$ such that
$Sx_n\to y$ as $n\to\infty$. On the one hand, the continuity of the
left action of~$\mathscr{A}$ on~$X$ implies that $a\cdot Sx_n\to
a\cdot y$ as $n\to\infty$, while on the other the continuity of the
map~$\delta^1S$ given by~\eqref{lemmaIntDerivEq1} shows that $a\cdot
Sx_n = (\delta^1S)(a,x_n)\to(\delta^1S)(a,0) = 0$ as $n\to\infty$, so
that $a\cdot y =0$.

Conversely, suppose that $\mathfrak{S}(S)\subseteq
\ann_{\mathscr{A}}X$. A bilinear map is jointly continuous if (and
only if) it is separately continuous, and~\eqref{lemmaIntDerivEq1}
shows that~$\delta^1S$ is continuous in the first variable by the
continuity of the module map.  Now fix $a\in\mathscr{A}$, and consider
the left action $L_a\colon x\mapsto a\cdot x,\, X\to X,$ which is
bounded and linear.  The assumption means that $L_a[\mathfrak{S}(S)] =
\{0\}$, so the composite map $L_aS$ is bounded by
Proposition~\ref{sepspacebasics}\romanref{sepspacebasics2}%
--\romanref{sepspacebasics3}.  Hence~$\delta^1S$ is continuous in the
second variable because $(\delta^1S)(a,b) = (L_aS)b$ for each
$b\in\mathscr{A}$.
\end{proof}

For a Banach space~$Y$, we denote by~$\ell_\infty(\N,Y)$ the Banach
space of all bounded sequences in~$Y$.

\begin{lemma}\label{lemmaXAmodwithannY}
Let~$\mathscr{A}$ be an in\-finite-di\-men\-sional Banach algebra with
the trivial product, and let~$Y$ be a closed subspace of a Banach
space~$X$. Endow~$X$ with the trivial right action
of~$\mathscr{A}$, and  suppose that there exists a bounded, linear
injection from the quotient space~$X/Y$ into~$\ell_\infty(\N,Y)$. Then
there exists a Banach left module action of~$\mathscr{A}$ on~$X$ with
$\ann_{\mathscr{A}} X = Y$.  The converse is true in the case
where~$\mathscr{A}$ is separable.
\end{lemma}

\begin{proof}
Let $T\colon X/Y\to\ell_\infty(\N,Y)$ be a bounded, linear injection,
and set \[ L_n = \pi_nTQ\colon\ X\to Y\qquad (n\in\N), \] where
$Q\colon X\to X/Y$ is the quotient map and
$\pi_n\colon\ell_\infty(\N,Y)\to Y$ is the $n^{\text{th}}$ coordinate
projection. 
Since $\mathscr{A}$ is infinite-dimensional, we may choose a countable
biorthogonal system $(a_n,f_n)_{n\in\N}$ in
$\mathscr{A}\times\mathscr{A}^*$; that is, $a_n\in\mathscr{A}$ and
$f_n$ is a bounded, linear functional on~$\mathscr{A}$ such that
$\langle a_m, f_n\rangle = \delta_{m,n}$ (the Kronecker delta) for
each $m,n\in\N$. Then the series
\[ \sum_{n=1}^\infty\frac{\langle a,f_n\rangle}{2^n\|f_n\| (\|L_n\|+1)}
L_nx \] converges absolutely in~$Y$ for each $a\in\mathscr{A}$ and
$x\in X$, and its sum, which we shall denote by $a\cdot x$, has norm
at most $\|a\|\,\|x\|$. It is easy to see that $(a,x)\mapsto a\cdot
x,\, \mathscr{A}\times X\to Y,$ is bilinear and that $a\cdot y =0$ for
each $y\in Y$, so that $a\cdot (b\cdot x) = 0 = (ab)\cdot x$ for each
$a,b\in\mathscr{A}$ and $x\in X$. Hence we have defined a Banach left
$\mathscr{A}$-module action on~$X$, and $Y\subseteq\ann_{\mathscr{A}}
X$. To prove the reverse inclusion, suppose that
$x\in\ann_{\mathscr{A}} X$. Then we have
\[ 0 = a_m\cdot x = \frac{1}{2^m\|f_m\|(\|L_m\|+1)}L_m x =
\frac{1}{2^m\|f_m\|(\|L_m\|+1)}\pi_mTQx \qquad (m\in\N), \] so that $TQx =0$,
and therefore $x\in\ker Q= Y$ because $T$ is injective.

To prove the converse statement in the case where $\mathscr{A}$ is
separable, suppose that $X$ has a Banach left $\mathscr{A}$-module
action with $\ann_{\mathscr{A}} X = Y$, and choose a sequence
$(a_n)_{n\in\N}$ of unit vectors in~$\mathscr{A}$ with $\clspa \{a_n :
n\in\N\} = \mathscr{A}$.  For each $a\in\mathscr{A}$ and $x\in X$, we
have $a\cdot x\in\ann_{\mathscr{A}} X = Y$ because $\mathscr{A}$ has
the trivial product, so
\[ T\colon x\mapsto  (a_n\cdot x)_{n\in\N},\quad
X\to\ell_\infty(\N,Y),\] defines a bounded, linear map with 
\[ \ker T = \{ x\in X: a_n\cdot x =0\ \text{for each}\ n\in\N\} =
\ann_{\mathscr{A}}X = Y \] because $\clspa \{a_n : n\in\N\} =
\mathscr{A}$. Hence the conclusion follows from the Fundamental
Isomorphism Theorem.
\end{proof}

The following corollary summarizes what we have found so far, and what
remains to be done to achieve our goal.

\begin{corollary}\label{takingstock}
Let~$\mathscr{A}$ be an infinite-dimensional Banach algebra with the
trivial product, let~$Y$ be a closed subspace of a Banach space~$X$
such that there exists a bounded, linear injection from the quotient
space~$X/Y$ into~$\ell_\infty(\N,Y),$ and suppose that there exists a
linear map $S\colon\mathscr{A}\to X$ with $\mathfrak{S}(S)\subseteq
Y$.  Then~$\mathscr{A}$ has a singular, admissible extension which
splits algebraically. Further, this extension splits strongly if and
only if
\begin{equation}\label{takingstockEq1} (S+T)[\mathscr{A}]\subseteq 
Y \end{equation} for some bounded, linear map $T\colon\mathscr{A}\to
X$.
\end{corollary}

\begin{proof} By Lemma~\ref{lemmaXAmodwithannY}, we can
  endow~$X$ with a Banach $\mathscr{A}$-bimodule structure such that
  the right action is trivial and $\ann_{\mathscr{A}}X = Y$, and
  Lemma~\ref{lemmaIntDeriv}\romanref{lemmaIntDeriv2} then implies
  that~$S$ is an intertwining map.  Hence $\mathscr{A}\oplus X$ is a
  Banach algebra with respect to the
  product~\eqref{AoplusXmult}, and~\eqref{extAoplusX} is a
  singular, admissible extension of~$\mathscr{A}$ which splits
  algebraically. As we stated, this extension splits strongly if and
  only if $S-D$ is continuous for some deri\-va\-tion
  \mbox{$D\colon\mathscr{A}\to X$.} By
  Lemma~\ref{lemmaIntDeriv}\romanref{lemmaIntDeriv1}, the latter
  condition is equivalent to~\eqref{takingstockEq1} (take $T = D-S$).
\end{proof}

We now come to our main technical lemma, which will ensure that we can
apply Corollary~\ref{takingstock} whenever the Banach
algebra~$\mathscr{A}$ satisfies the additional hypothesis that every
bounded, linear map from~$\mathscr{A}$ into~$\ell_1$ is compact.  The
statement of this lemma involves the following notation and
terminology.

First, the \emph{density character} of a Banach space~$X$ is the smallest
cardinality of a dense subset of~$X$. 

Second, for a non-empty set $\Xi$, we denote by~$\ell_1(\Xi)$ the
Banach space of all absolutely summable, scalar-valued functions
defined on~$\Xi$.  For $\xi\in\Xi$, we write $e_\xi$ for the function
which takes the value~$1$ at~$\xi$ and~$0$ elsewhere, and we write
$e_\xi'$ for the corresponding coordinate functional, so that $e_\xi'$
is given by $\langle f, e_\xi'\rangle = f(\xi)$ for each
$f\in\ell_1(\Xi)$. As usual, we write~$\ell_1$ for~$\ell_1(\N)$.

\begin{lemma}\label{keysepspacelemma}
Let $Z$ be an infinite-dimensional Banach space such that every
bounded, linear map from~$Z$ into~$\ell_1$ is compact, let $\Xi$ be a
normalized Hamel basis for~$Z,$ and consider the linear map $S\colon
Z\to\ell_1(\Xi)$ given by $S\xi = e_\xi$ for each $\xi\in\Xi$. Then
\begin{equation}\label{keysepspacelemmaEq1}
 (S+T)[Z]\not\subseteq\mathfrak{S}(S)
\end{equation} 
for each bounded, linear map $T\colon Z\to\ell_1(\Xi),$ and the
density character of the quotient space $\ell_1(\Xi)/\mathfrak{S}(S)$
is no greater than the density character of~$Z$.
\end{lemma}

\begin{proof} No infinite-dimensional Banach space has a countable
  Hamel basis, so the set~$\Xi$ is necessarily uncountable.  Assume
  towards a contradiction that $(S+T)[Z]\subseteq\mathfrak{S}(S)$ for
  some bounded, linear map $T\colon Z\to\ell_1(\Xi)$. 

To verify that~$T$ is compact, consider a bounded sequence
$(z_n)_{n\in\N}$ in~$Z$. Then the set
\[ \Gamma = \bigcup_{n\in\N}\{\xi\in\Xi : \langle Tz_n, e_\xi'\rangle\ne 0\} \]
is countable. Let $P\colon\ell_1(\Xi)\to\ell_1(\Gamma)$ be the
canonical projection.  By the hypothesis, $PT$ is compact, so
$(z_n)_{n\in\N}$ has a subsequence $(z_{n_j})_{j\in\N}$ such that
$(PTz_{n_j})_{j\in\N}$ is convergent. Since
\[ \| Tz_m - Tz_n \| = \| PTz_m - PTz_n\|\qquad (m,n\in\N) \] by the
choice of~$\Gamma$, we conclude that $(Tz_{n_j})_{j\in\N}$ is also
convergent, and hence~$T$ is compact.

In particular, $T$ has separable range and is strictly singular (see,
\emph{e.g.,} \cite[Propositions~1.11.7 and~1.11.9]{pietsch}), so
$T[Z]\subseteq\clspa\{e_\xi : \xi\in\Xi_0\}$ for some countable
subset~$\Xi_0$ of~$\Xi$, and the infinite-dimensional subspace
$\spa(\Xi\setminus\Xi_0)$ of~$Z$ contains a unit vector~$z$ such that
\mbox{$\|Tz\|\le1/2$}. Take a finite subset~$\Upsilon$
of~$\Xi\setminus\Xi_0$ such that $z = \sum_{\xi\in\Upsilon}s_\xi\xi$
for some scalars $s_\xi\ (\xi\in\Upsilon)$. By the assumption,
$(S+T)z\in\mathfrak{S}(S)$, so we can find a null sequence
$(x_n)_{n\in\N}$ in~$Z$ such that $Sx_n\to (S+T)z$ as
$n\to\infty$. Write $x_n$ as $x_n = \sum_{\xi\in\Upsilon} s_{n,\xi}\xi
+ y_n$, where $s_{n,\xi}\in\K\ (\xi\in\Upsilon)$ and
$y_n\in\spa(\Xi\setminus\Upsilon)$. For each $\xi\in\Upsilon$, we
have \[ Sy_n\in\spa\{e_\eta : \eta\in\Xi\setminus\Upsilon\}\subseteq\ker
e_\xi' \] by the definition of~$S$, and therefore
\[ s_{n,\xi} = \langle Sx_n,e_\xi'\rangle\to \langle 
(S+T)z,e_\xi'\rangle = \langle Sz, e_\xi'\rangle + \langle
Tz,e_\xi'\rangle = s_\xi\quad\text{as}\quad n\to\infty \] by the
continuity of~$e_\xi'$ and the fact that $Tz\in \clspa\{e_\eta :
\eta\in\Xi_0\}\subseteq\ker e_\xi'$. 
Using that the set~$\Upsilon$ is
finite, we obtain
\begin{equation}\label{limityn} 
 y_n = x_n - \sum_{\xi\in\Upsilon} s_{n,\xi}\xi\to 0 -
 \sum_{\xi\in\Upsilon} s_\xi\xi = -z\quad\text{as}\quad
 n\to\infty \end{equation} and
\begin{equation}\label{limitSyn} 
Sy_n = Sx_n - \sum_{\xi\in\Upsilon} s_{n,\xi}e_\xi\to (S+T)z -
\sum_{\xi\in\Upsilon} s_\xi e_\xi = Tz\quad\text{as}\quad
n\to\infty. \end{equation} For each $n\in\N$, write $y_n$ as $y_n =
\sum_{\xi\in\Xi\setminus\Upsilon} t_{n,\xi}\xi$, where only finitely
many of the scalars $t_{n,\xi}$ are non-zero. Then by the
subadditivity of the norm on~$Z$, we have
\[ \|y_n\|\le \sum_{\xi\in\Xi\setminus\Upsilon} |t_{n,\xi}| = 
\biggl\|\sum_{\xi\in\Xi\setminus\Upsilon} t_{n,\xi}e_\xi\biggr\| =
\|Sy_n\|\to \|Tz\|\le\frac12\quad\text{as}\quad n\to\infty, \]
using~\eqref{limitSyn}. This, however, contradicts that
$\|y_n\|\to\|-z\|= 1$ as $n\to\infty$ by~\eqref{limityn}, and
consequently~\eqref{keysepspacelemmaEq1} follows.

To prove the final clause, let
$Q\colon\ell_1(\Xi)\to\ell_1(\Xi)/\mathfrak{S}(S)$ be the quotient
map. Proposition~\ref{sepspacebasics} implies that~$QS$ is bounded, so
it will suffice to show that the range of~$QS$ is dense
in~$\ell_1(\Xi)/\mathfrak{S}(S)$. Let $\epsilon >0$. Each
element $y\in\ell_1(\Xi)/\mathfrak{S}(S)$ has the form $y = Qf$ for some
$f\in\ell_1(\Xi)$. Take a finite subset~$\Upsilon$ of~$\Xi$ such that
$\| f - \sum_{\xi\in\Upsilon}f(\xi)e_\xi\|\le\epsilon$, and set $z =
\sum_{\xi\in\Upsilon}f(\xi)\xi\in Z$. Then we have
\[ \|y - QSz\|\le \|f- Sz\| = \biggl\| f - 
\sum_{\xi\in\Upsilon}f(\xi)e_\xi\biggr\|\le\epsilon,\] from which the
conclusion follows.
\end{proof}

\begin{remark}
Let $Z$ be an infinite-dimensional Banach space which contains no
subspace isomorphic to~$\ell_1$. Combining Rosenthal's
$\ell_1$-theorem with the Schur property of~$\ell_1$, we deduce that
every bounded, linear map from~$Z$ into~$\ell_1$ is compact, and hence
the hypothesis of Lemma~\ref{keysepspacelemma} is satisfied.

We note in passing that the converse of this statement is not
true. Indeed, the Banach space~$\ell_\infty$ contains a subspace which
is (isometrically) isomorphic to~$\ell_1$, and every bounded, linear
map \mbox{$T\colon\ell_\infty\to\ell_1$} is compact. To verify the
latter fact, we observe that each such map~$T$ is weakly compact by a
theorem of Pe\l{}czy\'{n}ski~\cite{pelczynski} because no subspace of
its codomain is isomorphic to~$c_0$.  Using once more that~$\ell_1$
has the Schur property, we conclude that~$T$ is compact.
\end{remark}

\begin{corollary}\label{takingstockpart2}
Let~$\mathscr{A}$ be an infinite-dimensional, separable Banach algebra
with the trivial product, and suppose that every bounded, linear map
from~$\mathscr{A}$ into~$\ell_1$ is compact.  Then~$\mathscr{A}$ has a
singular, admissible extension which splits algebraically, but not
strongly.
\end{corollary}

\begin{proof} Applying Lemma~\ref{keysepspacelemma} with $Z=
  \mathscr{A}$ and taking $X = \ell_1(\Xi)$, where~$\Xi$ is a
  normalized Hamel basis for~$\mathscr{A}$, we obtain a linear map
  $S\colon\mathscr{A}\to X$ such that the closed subspace $Y =
  \mathfrak{S}(S)$ of~$X$ satisfies $(S+T)[\mathscr{A}]\not\subseteq
  Y$ for each bounded, linear map $T\colon\mathscr{A}\to X$, and the
  quotient space~$X/Y$ is separable. Hence~$X/Y$ embeds
  into~$\ell_\infty$, and thus into~$\ell_\infty(\N,Y)$, and the
  conclusion follows from Corollary~\ref{takingstock}.
\end{proof}

\begin{proof}[Proof of Theorem~{\normalfont{\ref{mainthm}}}]
Our strategy is to apply
Proposition~\ref{pullbacksandsplittingsofextensions} with $\mathscr{B}
= \mathscr{B}(E_{\text{R}})$, $\mathscr{C} = \ell_2^{\sim}$, and
$\beta = \psi$, using the notation of
Theorem~\ref{WEBEsplitexact}. By~\eqref{WEBEsplitexactEq1}, we have an
extension of the form~\eqref{ext2} which splits strongly.

The Schur property of~$\ell_1$ implies that every bounded, linear map
from~$\ell_2$ into~$\ell_1$ is compact, so by
Corollary~\ref{takingstockpart2}, we obtain a singular, admissible
extension of~$\ell_2$
\begin{equation}\label{ext4} \spreaddiagramcolumns{2ex}%
    \xymatrix{\{0\}\ar[r] & \ker\alpha_0\ar[r] &
      \mathscr{E}\ar^-{\displaystyle{\alpha_0}}[r] & \ell_2\ar[r] &
      \{0\}}
\end{equation} 
which splits algebraically, but not strongly. Passing to the
unitizations, and writing~$\mathscr{A}$ for the unitization of the
Banach algebra~$\mathscr{E}$ (so that $\mathscr{A} =
\mathscr{E}\oplus\K1$, with the product and norm defined in the usual
way), we obtain an extension of the form~\eqref{ext1} of $\mathscr{C}
= \ell_2^{\sim}$, where \[ \alpha(a+s1) = \alpha_0(a) + s1\qquad
(a\in\mathscr{E},\, s\in\K), \] and this extension clearly inherits
the properties of~\eqref{ext4}, so that it is singular and admissible
and splits algebraically, but it does not split strongly. Thus
Proposition~\ref{pullbacksandsplittingsofextensions} produces an
extension~\eqref{ext3} of $\mathscr{B} = \mathscr{B}(E_{\text{R}})$
which is singular and admissible and splits algebraically, but it does
not split strongly. To complete the proof, we note that the
algebra~$\mathscr{D}$ in~\eqref{ext3} is unital because the
algebras~$\mathscr{A}$, $\mathscr{B}$, and~$\mathscr{C}$
in~\eqref{ext1} and~\eqref{ext2} are unital, and hence so are the
surjections~$\alpha$ and~$\beta$.
\end{proof}

\begin{remark}
The Banach space~$E_{\text{DLW}}$ constructed by Dales, Loy, and
Willis~\cite{dlw} provides an interesting contrast to Read's
space~$E_{\text{R}}$, especially in relation to
Theorem~\ref{mainthm}. Indeed, $E_{\text{DLW}}$ shares
with~$E_{\text{R}}$ the property that~$\mathscr{B}(E_{\text{DLW}})$
admits a discontinuous algebra homomorphism into a Banach algebra
(under the assumption of the Continuum Hypothesis), but it differs in
that every derivation from~$\mathscr{B}(E_{\text{DLW}})$ into a Banach
$\mathscr{B}(E_{\text{DLW}})$-bi\-module is continuous. Hence every
singular, admissible extension of~$\mathscr{B}(E_{\text{DLW}})$ which
splits algebraically also splits strongly by a general result that was
stated in the Introduction.
\end{remark}

\bibliographystyle{amsplain}

\end{document}